\documentclass[12pt,onecolumn]{IEEEtran}
\everymath{\displaystyle}
\usepackage{graphicx}
\usepackage{epsfig}
\usepackage{float}
\usepackage{amsmath}
\usepackage{amssymb}
\usepackage{amsthm}
\usepackage{booktabs}
\usepackage{subfigure}
\usepackage{geometry}
\usepackage{bm}

\newtheorem{thm}{Theorem}

\newtheorem{remark}{Remark}
\newtheorem{assump}{Assumption}

\begin{document}
\title{Adaptive High Order Sliding Mode Observer Based Fault Reconstruction
for a Class of Nonlinear Uncertain Systems : Application to PEM Fuel Cell System}
\author{Jianxing Liu, Salah Laghrouche and Maxime Wack\\
Laboratoire IRTES, Universit\'e de Technologie de Belfort-Montb\'eliard, Belfort, France.
\\ jiang-xing.Liu@utbm.fr; salah.laghrouche@utbm.fr; maxime.wack@utbm.fr}
\maketitle
\begin{abstract}                
\boldmath
This paper focuses on observer based fault reconstruction for
a class of nonlinear uncertain systems with Lipschitz nonlinearities.
An adaptive-gain Super-Twisting (STW) observer is developed for observing the system states, where the adaptive law compensates the uncertainty in parameters. The inherent equivalent output error injection feature of STW algorithm is then used to reconstruct the fault signal.
The performance of the proposed observer is validated through a Hardware-In-Loop (HIL)
simulator which consists of a commercial twin screw compressor and
a real time Polymer Electrolyte Membrane fuel cell emulation system.
The simulation results illustrate the feasibility and effectiveness
of the proposed approach for application to fuel cell systems.
\end{abstract}
\section{INTRODUCTION}
Many physical systems cannot be operated safely without reliable Fault Detection and Isolation (FDI) schemes in place. Such systems include Fuel Cells, which use Hydrogen as fuel, and system failure or mechanical faults can lead to catastrophic consequences. FDI is usually achieved by generating residual signals, obtained from the difference between the actual system outputs and their estimated values calculated from dynamic models. Such approach usually involves two steps: the first step is
to decouple the faults of interest from uncertainties and the second step
is to generate residual signals and detect faults by decision logic.
Several practical techniques for these steps have been proposed in contemporary literature,
for example geometric approaches \cite{de2001geometric},
$H_\infty$-optimization technique \cite{qiu1993robust,hou1996lmi},
observer based approaches
(e.g. adaptive observers \cite{yang1995nonlinear,wang1997use},
high gain observers \cite{besanccon2003high},
unknown input observers \cite{hou1992design,chen1996design,chen1999robust}).
However, in active fault tolerant control (FTC) systems
\cite{zhang2008bibliographical},
not only the fault needs to be detected and isolated but also needs
to be estimated so that its effect can be compensated by reconfiguring
the controller \cite{edwards2006sensor,alwi2008fault}.
Hence, there is need for fault reconstruction schemes which estimate
the fault's shape and magnitude.

Sliding mode technique is known for its insensitivity to external disturbances, high accuracy and finite time convergence.
Sliding mode observers have been widely used
for fault reconstruction.
\cite{edwards2000sliding} proposed a fault reconstruction approach based on equivalent output error injection. In this method, the resulting residual signal can approximate the actuator fault to any required accuracy.
Based on the work of \cite{edwards2000sliding}, \cite{tan2002sliding} proposed a sensor fault reconstruction method for well-modeled linear systems by considering a linear matrix inequality (LMI) observer design. This approach is impractical, as there is no explicit consideration of disturbance or uncertainty. To overcome this, \cite{tan2003sliding} proposed an FDI scheme for a class of linear systems with uncertainty, using LMI for minimizing the $L_2$ gain between the uncertainty and the fault reconstruction signal.
It should be noted that, only linear systems are considered in the above works,
and only few works have been reported for nonlinear systems.
\cite{jiang2004fault} proposed a sliding mode observer based fault estimation approach for a class of nonlinear systems with uncertainties.
\cite{yan2007nonlinear} proposed a precise fault reconstruction scheme, based on equivalent output error injection,
for a class of nonlinear systems with uncertainty. A sufficient condition based on LMI is presented for the existence and stability of a robust sliding mode observers.
The limitation is that requires a strong structural condition of the distribution associated with uncertainties.
This structural constraint is relaxed in \cite{yan2008adaptive}, where the fault distribution vector and the structure matrix of the uncertainty are allowed to be functions of the system output and input.
However, in these papers involving uncertainty and fault,
most of all require that the uncertainty and fault are bounded with known bounds.
Although in the job of \cite{yan2008adaptive}, it is not necessary to known the bounds on uncertainty \emph{a priori}, but it still requires that the boundary of fault signal is known.

From the application point of view, fuel cell systems have also been the subject of many FDI studies. \cite{ingimundarson2008model} developed a hydrogen leak detection method without the use of relative humidity sensors. \cite{escobet2009model} designed a fault diagnosis methodology for PEMFCs where the residuals are generated from the differences between the PEM simulator included with a set of typical faults and non-faulty fuel cell model. \cite{de2011lpv} proposed a Linear Parameter Varying (LPV) observer based fault detection approach for PEMFC, the residuals are defined as the differences between the process measurements and the outputs of a LPV Luenberger observer.
The above works have not considered the design of state estimation, parameter identification and directly fault reconstruction, simultaneously.
Even more, the models used in \cite{de2011lpv,pukrushpan2004control} are obtained through a Jacobian linearization of the PEMFC nonlinear dynamic model around the optimal operating point.
It is well known that the fuel cell system exhibits highly nonlinear dynamics and a fuel cell's efficiency is highly dependent on the operating conditions, such as temperature, humidity, and air flow. A set of auxiliary elements (valves, compressor, sensors, etc.) are needed to make the fuel cell work at the optimal operating point. The problems of state estimation and fault detection arise because of incomplete knowledge of the parameter and states of the system.

In this paper, fault reconstruction is studied for a class of nonlinear uncertain systems
with Lipschitz nonlinearities. This class of systems sufficiently defines the nonlinear dynamics of the Fuel Cell system. The uncertain parameter is estimated by a simple adaptive update law.
Estimate of the uncertain parameter is then injected into a Second Order Sliding Mode (SOSM) observer based on Super Twisting algorithm, which maintains a sliding motion even in the presence of faults. The faults are reconstructed by analyzing the equivalent output error injection
\cite{edwards2000sliding,spurgeon2008sliding}.
We extend the result of \cite{yan2008adaptive} to a class of nonlinear uncertain systems with Lipschitz nonlinearities. The difficulty of the a-priori knowledge of fault signal bound, as encountered in \cite{edwards2000sliding}, is overcome by the gain adaptation with respect to observation errors. The gains of Super Twisting (STW) algorithm are allowed to adapt based on the 'quality' of the obtained sliding mode, and the relation of the gains are chosen to make the algorithm hold homogeneity. Lyapunov analysis of the system with the adaptive observer demonstrates that the algorithm establishes ideal SOSM, as compared to a similar work by \cite{Shtessel2012759}, in which real Sliding Mode is established.

The Fault reconstruction scheme is applied on a PEMFC system which belongs to a class of nonlinear uncertain systems with Lipschitz nonlinearities.
The nonlinear model of a fuel cell system, presented in
\cite{talj2010experimental}, is used. This model has been validated in a large operating range. The adaptive-gain STW sliding mode observer estimates the system states (oxygen and nitrogen partial pressures and compress speed),
and an adaptive update law is employed to estimate the stack current, which is considered as an uncertain parameter. This eliminates the need of an extra current sensor. The fault considered in this study is a mechanical failure in the air circuit, which results in an abnormal air flow.
The fault signal is reconstructed by analyzing the equivalent output error injection, which is obtained on-line from comparisons between the
measurements from the sensor installed in the real system
and the outputs of the observer system.
In order to verify the effectiveness of the proposed approach of the
application to a fuel cell system, Hardware-In-Loop (HIL) tests are conducted on a test bench which consists of a commercial twin screw compressor
and a real time fuel cell emulation system
\cite{lemevs2006dynamic,randolf2006test,matraji2013}.
The test bench, which is presented in this work, disposes a
HIL simulator consists of a real time fuel cell system and a twin screw compressor.
These test results show that the fault detection scheme successfully reconstructs the faults.
\section{PROBLEM FORMULATION}
Consider the following nonlinear system,
\begin{equation}\label{eqn:nonlinear system}
\left\{
\begin{split}
\dot{x} &=Ax+g(x,u)+\phi(y,u)\theta+\omega(y,u)f(t),\\
y &=Cx,
\end{split}
\right.
\end{equation}
where $x\in \mathbb{R}^n$ is the system state vector, $u(t)\in \mathbb{R}^m$ is the control input which is assumed to be known, $y\in \mathbb{R}^p$ is the output vector.
$g(x,u)\in \mathbb{R}^n$ is Lipschitz continuous,
$\phi(y,u)\in \mathbb{R}^{n\times q}$ and
$\omega(y,u)\in \mathbb{R}^{n\times r}$ are assumed to be some
smooth and bounded functions with $p\geq q+r$.
The unknown parameter vector $\theta\in \mathbb{R}^q$ is assumed to be  constant and
$f(t)\in \mathbb{R}^r$ is smooth fault signal vector, which satisfies
\begin{equation}\label{condition:1}
\begin{split}
\|{f}(t)\|\leq \rho_1,\quad \|{\dot{f}}(t)\|\leq \rho_2.
\end{split}
\end{equation}
where $\rho_1,\rho_2$ are some positive constants that might be known or unknown.

Assume that $(A,C)$ is an observable pair, and there exists
a linear coordinate transformation
$z=Tx=\begin{bmatrix}
I_p&0\\
-H_{(n-p)\times p}&I_{n-p}
\end{bmatrix}
x=
\begin{bmatrix}
z^T_1&z^T_2
\end{bmatrix}
^T$, with $z_1\in \mathbb{R}^{p}$ and $z_2\in \mathbb{R}^{n-p}$, such that
\begin{itemize}
  \item $TAT^{-1}=\begin{bmatrix}
A_{11}&A_{12}\\A_{21}&A_{22}
\end{bmatrix}$, where the matrix $A_{22}
\in \mathbb{R}^{(n-p)\times (n-p)}$ depends on the gain matrix
$H_{(n-p)\times p}$ is Hurwitz stable.
  \item $CT^{-1}=
  \begin{bmatrix}
  I_p&0
  \end{bmatrix}
  $, where $I_p\in \mathbb{R}^{p\times p}$ is an identity matrix.
\end{itemize}
\begin{assump}\label{assump1}
There exists function $\omega_1(y,u)$ such that
\begin{equation}\label{fault distribution}
T\omega(y,u)=
\begin{bmatrix}
\omega_1(y,u)\\
0
\end{bmatrix}
\end{equation}
where $\omega_1(y,u)\in \mathbb{R}^{p\times r}$.
\end{assump}

System (\ref{eqn:nonlinear system}) is described by the following equations
in the new coordinate system,
\begin{equation}\label{eqn:z system}
\left\{
\begin{split}
\dot{z}=&TAT^{-1}z+Tg(T^{-1}z,u)+T\phi(y,u)\theta+T\omega(y,u)f(t),\\
y=&CT^{-1}z,
\end{split}
\right.
\end{equation}
By reordering the state variables, system (\ref{eqn:z system}) can be rewritten as
\begin{equation}\label{eqn:zsystem}
\left\{
\begin{split}
\dot{y}=&{A}_{11}y+{A}_{21}z_2
+g_1(z_2,y,u)+{\phi}_1(y,u)\theta+{\omega}_1(y,u)f(t),\\
\dot{z}_2=&{A}_{22}z_2+{A}_{21}y
+{g}_2(z_2,y,u)+{\phi}_2(y,u)\theta,\\
y=&z_1,
\end{split}
\right.
\end{equation}
where
\begin{equation}\label{eqn:sub system}
\begin{split}
T\phi(y,u)=&
\begin{bmatrix}
\phi_1(y,u)\\
\phi_2(y,u)
\end{bmatrix}
,\quad
    Tg(T^{-1}z,u)=
\begin{bmatrix}
g_1(z_2,y,u)\\
g_2(z_2,y,u)
\end{bmatrix}.
\end{split}
\end{equation}
$\phi_1:\mathbb{R}^{p}\times \mathbb{R}^m\rightarrow \mathbb{R}^{p},\ \
\phi_2:\mathbb{R}^{p}\times \mathbb{R}^m\rightarrow \mathbb{R}^{n-p},
\ \
g_1:\mathbb{R}^{p}\times \mathbb{R}^{n-p}\times \mathbb{R}^m\rightarrow \mathbb{R}^p,\ \
g_2:\mathbb{R}^{p}\times \mathbb{R}^{n-p}\times \mathbb{R}^m\rightarrow \mathbb{R}^{n-p}$.
\section{ADAPTIVE HIGH ORDER SLIDING MODE OBSERVER DESIGN}\label{sec:3}
Now, we consider the problem of an adaptive-gain STW observer for system (\ref{eqn:z system}),
in which the uncertain parameters are estimated with the help of an adaptive law.
Then, using the equivalent output injection of the observer,
we will develop the precise fault reconstruction method.
The basic assumption on the class of nonlinear systems under consideration
is as follows:
\begin{assump}\label{assump:assump4}
There exists a nonsingular matrix $\bar{T}\in \mathbb{R}^{p\times p}$, such that
\begin{equation}\label{eqn:W condition}
\bar{T}\begin{bmatrix}
{\phi}_1(y,u)&\omega_1(y,u)
\end{bmatrix}=
\begin{bmatrix}
{\Phi}_1(y,u)&0_{q\times r}\\
0_{r\times q}&{\Phi}_2(y,u)
\end{bmatrix}
\end{equation}
where ${\Phi}_1(y,u)\in \mathbb{R}^{q\times q},\ \ {\Phi}_2(y,u)\in \mathbb{R}^{r\times r}$
are both nonsingular matrices.
\end{assump}
\begin{remark}
The main limitation in the Assumption (\ref{eqn:W condition})
is that the matrices
$\begin{bmatrix}
{\phi}_1(y,u)&\omega_1(y,u)
\end{bmatrix}$
must be block-diagonalizable by elementary row transformations.
For simplicity, the case $p=q+r$ is considered,
the same results can be obtained in the case $p>q+r$
\cite{yan2008adaptive}.
\end{remark}
Let $z_y=\bar{T}y$, where $\bar{T}$ is defined in Assumption (\ref{assump:assump4}).
Then, System (\ref{eqn:z system}) can be described by,
\begin{equation}\label{eqn:two output system}
\left\{
\begin{split}
\dot{z}_{y}=&\bar{T}A_{11}\bar{T}^{-1}z_y+\bar{T}{A}_{21}z_2
+\bar{T}g_1(y,z_2,u)+
\begin{bmatrix}
{\Phi}_1(y,u)\\
0
\end{bmatrix}
\theta
+
\begin{bmatrix}
0\\
{\Phi}_2(y,u)
\end{bmatrix}f(t),\\
\dot{z}_2=&{A}_{22}z_2+{A}_{21}y
+{g}_2(z_2,y,u)+{\phi}_2(y,u)\theta,\\
y=&\bar{T}^{-1}z_y,
\end{split}
\right.
\end{equation}
where
\begin{equation}\label{matrix:transofrm matrix}
\begin{split}
&\bar{T}\cdot {A}_{11}=
\begin{bmatrix}
\bar{A}_{11}\\
\bar{A}_{12}
\end{bmatrix}
,\ \
\bar{T}\cdot {A}_{12}=
\begin{bmatrix}
\bar{A}_{21}\\
\bar{A}_{22}
\end{bmatrix},
\\
&\bar{T}\cdot {g}_1(y,z_2,u)=
\begin{bmatrix}
{W}_{g_1}(y,z_2,u)\\
{W}_{g_2}(y,z_2,u)
\end{bmatrix}.
\end{split}
\end{equation}
We define
$z_y=\begin{bmatrix}z_{y_1}&z_{y_2}\end{bmatrix}^T$,
where $z_{y_1}\in \mathbb{R}^{q}, z_{y_2}\in \mathbb{R}^{r}$.
Then, from (\ref{eqn:two output system}) and \eqref{matrix:transofrm matrix}, we obtain
\begin{equation}\label{eqn:1}
\left\{
\begin{split}
\dot{z}_{y_1}=&\bar{A}_{11}y+\bar{A}_{21}z_2
+{W}_{g_1}(y,z_2,u)+{\Phi}_1(y,u)\theta,\\
\dot{z}_{y_2}=&\bar{A}_{12}y+\bar{A}_{22}z_2
+{W}_{g_2}(y,z_2,u)+{\Phi}_2(y,u)f(t),\\
\dot{z}_2=&{A}_{21}y+{A}_{22}z_2
+{g}_2(y,z_2,u)+{\phi}_2(y,u)\theta,\\
y=&T^{-1}
\begin{bmatrix}
z_{y_1}&z_{y_2}
\end{bmatrix}^T,
\end{split}
\right.
\end{equation}

The observer is represented by the following dynamical system
\begin{equation}\label{eqn:two output system observer}
\left\{
\begin{split}
\dot{\hat{z}}_{y_1}=&\bar{A}_{11}y+\bar{A}_{21}\hat{z}_2
+{W}_{g_1}(y,\hat{z}_2,u)+{\Phi}_1(y,u)\hat{\theta}+\mu(e_{y_1}),\\
\dot{\hat{z}}_{y_2}=&\bar{A}_{12}y+\bar{A}_{22}\hat{z}_2
+{W}_{g_2}(y,\hat{z}_2,u)+\mu(e_{y_2}),\\
\dot{\hat{z}}_2=&{A}_{21}y+{A}_{22}\hat{z}_2
+{g}_2(y,\hat{z}_2,u)+{\phi}_2(y,u)\hat{\theta},\\
\hat{y}=&T^{-1}
\begin{bmatrix}
\hat{z}_{y_1}&\hat{z}_{y_2}
\end{bmatrix}^T,
\end{split}
\right.
\end{equation}
where the continuous output error injection $\mu(s)$ is given by the
super-twisting algorithm \cite{levant1993sliding}:
\begin{equation}\label{eqn:two stw law}
\left\{
\begin{split}
\mu(s)&=\lambda(t) |s|^{\frac{1}{2}}sign(s)+\varphi(s),\\
\dot{\varphi}(s)&=\alpha(t)sign(s),
\end{split}
\right.
\end{equation}
with the gains $\lambda(t),\alpha(t)$
are functions of time and are explained later on in Theorem 3.3. The observation errors are then defined as
$e_{y_1}=z_{y_1}-\hat{z}_{y_1},\ \ e_{y_2}=z_{y_2}-\hat{z}_{y_2}, e_2=z_2-\hat{z}_2,
\ \ \tilde{\theta}=\theta-\hat{\theta}$.
The estimate of $\theta$, denoted by $\hat{\theta}$, is given by the following adaptive law:
\begin{equation}\label{eqn:adaptive laws}
\begin{aligned}
\dot{\hat{\theta}}=&-K(y,u)\left(
\bar{A}_{11}y+\bar{A}_{21}\hat{z}_2+W_{g_1}(y,\hat{z}_2,u)+\Phi_1(y,u)\hat{\theta}
-\dot{z}_{y_1}\right),
\end{aligned}
\end{equation}
where $K(y,u)\in \mathbb{R}^{q\times q}$ is a matrix design parameter which will
be determined later.
\begin{remark}
It can be seen that the adaptive law (\ref{eqn:adaptive laws})
depends upon $\dot{z}_{y_1}$.
A real time robust exact differentiator
proposed in \cite{levent1998robust} can be used to estimate the time derivative of
${z}_{y_1}$ in finite time.
The differentiator has the following form
\begin{equation}\label{differentiator}
\left\{
\begin{split}
\dot{z}_0 =& -\lambda_0L_0^{\frac{1}{2}}|z_0-z_{y_1}|^{\frac{1}{2}}sign(z_0-z_{y_1})+z_1,\\
\dot{z}_1 =& -\alpha_0L_0sign(z_0-z_{y_1}),
\end{split}
\right.
\end{equation}
where $z_0$ and $z_1$ are the real time estimations of $z_{y_1}$ and
$\dot{z}_{y_1}$ respectively.
The parameters of the differentiator $\lambda_0=1, \alpha_0=1.1$
are suggested in \cite{levent1998robust}.
$L_0$ is the only parameter needs to be tuned
according to the condition $|\ddot{z}_{y_1}|\leq L_0$.
\end{remark}
It follows from (\ref{eqn:1}) and (\ref{eqn:two output system observer}, \ref{eqn:adaptive laws})
that the error dynamical equation is
\begin{eqnarray}
\dot{e}_2&=&{A}_{22}e_2+\tilde{g}_2(y,{z}_2,\hat{z}_2,u)
+{\phi}_2(y,u)\tilde{\theta},
\label{error1}\\
\dot{\tilde{\theta}}&=&-K(y,u)\left(
\bar{A}_{21}e_2+\tilde{W}_{g_1}(y,{z}_2,\hat{z}_2,u)+
\Phi_1(y,u)\tilde{\theta}\right),
\label{error2}\\
\dot{e}_{y_1}&=&-\mu(e_{y_1})+\bar{A}_{21}e_2+
{\Phi}_1(y,u)\tilde{\theta}+\tilde{W}_{g_1}(y,{z}_2,\hat{z}_2,u),
\label{error3}\\
\dot{e}_{y_2}&=&-\mu(e_{y_2})+\bar{A}_{22}e_2+{\Phi}_2(y,u)f(t)
+\tilde{W}_{g_2}(y,{z}_2,\hat{z}_2,u),
\label{error4}
\end{eqnarray}
where $\tilde{g}_2(y,{z}_2,\hat{z}_2,u)={g}_2(y,{z}_2,u)-{g}_2(y,\hat{z}_2,u),
\tilde{W}_{g_1}(y,{z}_2,\hat{z}_2,u)=W_{g_1}(y,z_2,u)-W_{g_1}(y,\hat{z}_2,u)$ and \\
$\tilde{W}_{g_2}(y,{z}_2,\hat{z}_2,u)={W}_{g_2}(y,{z}_2,u)-{W}_{g_2}(y,\hat{z}_2,u)$. Some assumptions are imposed upon the system:
\begin{assump}\label{3.2}
The known nonlinear terms $g_2(y,{z}_2,u)$, $W_{g_1}(y,z_2,u)$
and $W_{g_2}(y,z_2,u)$ are Lipschitz continuous with respect to $z_2$   i.e.
\begin{equation}\label{eqn:new Lipschitz condition}
\begin{split}
&\|{g}_2(y,{z}_2,u)-{g}_2(y,\hat{z}_2,u)\|\leq {\gamma}_2\|{z}_2-\hat{z}_2\|, \\
&\|{W}_{g_1}(y,{z}_2,u)-{W}_{g_1}(y,\hat{z}_2,u)\|\leq {\gamma}_{g_1}\|{z}_2-\hat{z}_2\|,\\
&\|{W}_{g_2}(y,{z}_2,u)-{W}_{g_2}(y,\hat{z}_2,u)\|\leq {\gamma}_{g_2}\|{z}_2-\hat{z}_2\|.
\end{split}
\end{equation}
where $\gamma_{g_1}$, $\gamma_{g_2}$ and $\gamma_2$ are the known Lipschitez constants for
${W}_{g_1}(y,{z}_2,u)$, $W_{g_2}(y,{z}_2,u)$ and ${g}_2(y,{z}_2,u)$
respectively \cite{zhang2010fault}.
\end{assump}
\begin{assump}\label{3.3}
Suppose that the Huritwz matrix $A_{22}$ satisfies the following Riccati equation
\begin{equation}\label{Riccati equation}
    {A}^T_{22}P_1+P_1{A}_{22}+\gamma^2_2P_1P_1+2I_{n-p}+\varepsilon I_{n-p}=0,
\end{equation}
which has a symmetric positive-definite solution $P_1$
for some $\varepsilon>0$ \cite{rajamani1998observers}.
\end{assump}
\begin{assump}\label{3.4}
Suppose that the matrix design parameter $K(y,u)$ satisfies the following equation
\begin{equation}\label{matrix design parameter}
K(y,u){\Phi}_1(y,u)+{\Phi}^T_1(y,u)K^T(y,u)-\gamma^2_{g_1}K(y,u)K^T(y,u)-\epsilon I_{q}=0,
\end{equation}
for some $\epsilon>0$.
\end{assump}
Now, We will first consider the stability of the error systems (\ref{error1}) and (\ref{error2}).
\begin{thm}\label{thm1}
Consider the systems (\ref{error1}, \ref{error2}) satisfying the Assumptions
(\ref{3.2}, \ref{3.3} and \ref{3.4}).
Then, systems (\ref{error1}) and (\ref{error2}) are exponentially stable if the matrix
\begin{equation}\label{positive matrix Q}
Q=
\begin{bmatrix}
\varepsilon I_{n-p}&P_1{\phi}_2(y,u)-\bar{A}^T_{21}K^T(y,u) \\
{\phi}^T_2(y,u)P_1-K(y,u)\bar{A}_{21}&\epsilon I_{q}
\end{bmatrix}
\end{equation}
is positive definite, where $P_1$ and $K(y,u)$ satisfies (\ref{Riccati equation}) and
(\ref{matrix design parameter}) respectively.
\end{thm}
\begin{proof}
A candidate Lyapunov function is chosen as
\begin{equation}\label{eqn:lyapunov_3}
    V_1(e_2,\tilde{\theta})=e^T_2P_1e_2+\tilde{\theta}^T\tilde{\theta},
\end{equation}
and the time derivative of $V_1$ along the solution of the system (\ref{error1}, \ref{error2}) is given by
\begin{equation}\label{eqn:dlyapunov_3}
\begin{split}
\dot{V}_1=&e^T_2({A}^T_{22}P_1+P_1{A}_{22})e_2
+2e^T_2P_1\tilde{g}_2(y,{z}_2,\hat{z}_2,u)+2e^T_2P_1{\phi}_2(y,u)\tilde{\theta}
-2e^T_2\bar{A}^T_{21}K^T(y,u)\tilde{\theta}\\
&-2\tilde{\theta}^TK(y,u)\tilde{W}_{g_1}(y,{z}_2,\hat{z}_2,u)
-\tilde{\theta}^T\left(K(y,u){\Phi}_1(y,u)
+{\Phi}^T_1(y,u)K^T(y,u)\right)\tilde{\theta}\\
\leq&e^T_2({A}^T_{22}P_1+P_1{A}_{22}+\gamma^2_2P_1P_1+2I_{n-p})e_2+
2e^T_2\left(P_1{\phi}_2(y,u)-\bar{A}^T_{21}K^T(y,u)\right)\tilde{\theta}\\
&-\tilde{\theta}^T\left(K(y,u){\Phi}_1(y,u)
+{\Phi}^T_1(y,u)K^T(y,u)-\gamma^2_{g_1}K(y,u)K^T(y,u)\right)\tilde{\theta}\\
=&-
\begin{bmatrix}
e^T_2&\tilde{\theta}^T
\end{bmatrix}
Q
\begin{bmatrix}
e_2\\
\tilde{\theta}
\end{bmatrix},
\end{split}
\end{equation}
Hence, It follows that $e_2, \tilde{\theta}$ will converge to zero exponentially.
\end{proof}

Before proving the convergence of the systems (\ref{error3}) and (\ref{error4}),
we will show some conditions.
Theorem \ref{thm1} shows that
$\lim\limits_{t\rightarrow \infty}e_2(t)=0$ and
$\lim\limits_{t\rightarrow \infty}\tilde{\theta}(t)=0$.
Consequently, the errors $e_2, \tilde{\theta}$ and its derivatives
$\dot{e}_2$, $\dot{\tilde{\theta}}$ are bounded.
\begin{assump}\label{3.6}
The time derivative of functions $\Phi_1(y,u)$, $\Phi_2(y,u)$,
$W_{g_1}(y,z,u)$ and $W_{g_2}(y,z,u)$ are bounded.
\end{assump}
Therefore, under the Assumption \ref{3.6},
the time derivative of the nonlinear terms in the error dynamics
(\ref{error3}, \ref{error4}) are bounded.
That is,
\begin{equation}\label{ineqn:bounded error}
\begin{split}
    &\left\|\frac{d}{dt}\left(\bar{A}_{22}e_2+{\Phi}_2(y,u)f(t)
    +\tilde{W}_{g_2}(y,{z}_2,\hat{z}_2,u)\right)\right\|\leq \chi_1,\\
&\left\|\frac{d}{dt}\left(
\bar{A}_{21}e_2+
{\Phi}_1(y,u)\tilde{\theta}+\tilde{W}_{g_1}(y,{z}_2,\hat{z}_2,u)
\right)\right\|\leq \chi_2.
\end{split}
\end{equation}
where $\chi_1$ and $\chi_2$ are
some \textbf{\textit{unknown}} positive constants.

In what follows, the objective is to prove the finite time convergence of
the systems (\ref{error3}) and (\ref{error4}) based on adaptive-gain
STW algorithms \cite{Shtessel2012759,alwi2011oscillatory}.
\begin{thm}\label{thm2}
Consider the system (\ref{error3}, \ref{error4})
with the following adaptation laws,
\begin{equation}\label{eqn:adaptive gains}
\left\{
\begin{split}
\lambda(t)&=2\sqrt{L(t)},\\
\alpha(t)&=4L(t),
\end{split}
\right.
\end{equation}
where the dynamic of the positive time varying function $L(t)$ is given by
\begin{equation}\label{eqn:adaptive gains rt}
\dot{L}(t) =
\begin{cases}
  k, \quad\quad if \quad |e_{y_i}|\neq0\\
  0, \quad\quad else
\end{cases}
\end{equation}
where $k>0$ is a positive design constant. Then, all trajectories of the system (\ref{error3}, \ref{error4})
converge to zero in finite time.
\end{thm}
The proof of this theorem is similar to \cite{alwi2011oscillatory}.

\begin{remark}\label{time scale}
The advantage of the law (\ref{eqn:adaptive gains}) is that the uncertainty affecting the system can be attenuated by $\frac{1}{L}$, where $L$ is a positive constant.
Consider the structure of the system (\ref{error3}) and (\ref{error4}),
\begin{equation}\label{eqn:original sys}
\left\{
\begin{split}
\frac{de}{dt}=&-\lambda_0 |e|^{\frac{1}{2}}sign(e) + \varphi,\\
\frac{d\varphi}{dt}=&-\alpha_0 sign(e)+\dot{\nu},
\end{split}
\right.
\end{equation}
where $\nu$ denotes the uncertainty which affects the system (\ref{error3}) and (\ref{error4}).

Then according to \eqref{eqn:adaptive gains}, the time-scale is modified as
\begin{equation}\label{eqn:time scaling1}
  d\tau=Ldt,\quad \tilde{e}=Le.
\end{equation}
then, for every fixed value of $L$, the system (\ref{eqn:original sys}) is equivalent to the following system:
\begin{equation}\label{eqn:coordinate transform}
\left\{
\begin{split}
\frac{d\tilde{e}}{d\tau}=&-\lambda_0 |\tilde{e}|^{\frac{1}{2}}sign(\tilde{e}) + \varphi,\\
\frac{d\varphi}{d\tau}=&-\alpha_0 sign(\tilde{e})+\frac{\dot{\nu}}{L},
\end{split}
\right.
\end{equation}
Now, it follows from a comparison between the systems (\ref{eqn:original sys}) and (\ref{eqn:coordinate transform}), that the uncertainty effects on the system (\ref{eqn:coordinate transform}) will be attenuated by $\frac{1}{L}$ for every fixed positive value of $L$ with respect to the new time-scale $\tau$, without changing the system structure.
In our paper, $\lambda_0=2$ and $\alpha_0=4$ are chosen.
\end{remark}
Theorems \ref{thm1} and \ref{thm2} have shown that systems (\ref{eqn:two output system observer})
and (\ref{eqn:adaptive laws}) are an asymptotic
state observer and uncertain parameter observer for the system (\ref{eqn:1}) respectively.
In the next Section, we will develop the fault reconstruction approach
based on those two observers.
\section{Fault Reconstruction}
The fault signal $f(t)$ will be reconstructed based on the proposed observer
by using an equivalent output error injection which can be obtained
as soon as the sliding surface is reached.

It follows from Theorem \ref{thm2} that $e_{y_1},e_{y_2},\dot{e}_{y_1},\dot{e}_{y_2}$
in (\ref{error3}) and (\ref{error4})
are driven to zero in finite time.
Then, the equivalent output error injection are obtained directly
\begin{equation}\label{eqn:equivalent injections}
\mu(e_{y_2})=\bar{A}_{22}e_2+{\Phi}_2(y,u)f(t)
+\tilde{W}_{g_2}(y,{z}_2,\hat{z}_2,u).
\end{equation}
From Assumption \ref{assump:assump4},
${\Phi}_2(y,u)$ is a nonsingular matrix,
then the estimation of $f(t)$ can be constructed as
\begin{equation}\label{eqn:two estimation of f}
    \hat{f}(t)={\Phi}^{-1}_2(y,u)\mu(e_{y_2}).
\end{equation}
\begin{thm}\label{theorem:6}
If the conditions of Theorems (\ref{thm1}, \ref{thm2}) are satisfied,
then $\hat{f}(t)$ defined in (\ref{eqn:two estimation of f}) is a reconstruction
of the fault $f(t)$ since
\begin{equation}\label{eqn:theorem 6 fault error}
    \lim\limits_{t \to \infty }{\left\Vert f(t)-\hat{f}(t)\right\Vert}=0.
\end{equation}
\end{thm}
\begin{proof}
It follows from (\ref{eqn:equivalent injections}) and
(\ref{eqn:two estimation of f}) that
\begin{equation}\label{eqn: fault estimation error}
\begin{split}
&\left\Vert f(t)-\hat{f}(t)\right\Vert=
\left\Vert
{\Phi}^{-1}_2(y,u)\left(\bar{A}_{22}e_2+\tilde{W}_{g_2}\right)\right\Vert\\
&\leq
\left\Vert
{\Phi}^{-1}_2(y,u)\bar{A}_{22}
\right\Vert
\left\Vert e_2
\right\Vert
+\gamma_{g_2}\left\Vert
{\Phi}^{-1}_2(y,u)\right\Vert
\left\Vert e_2 \right\Vert.
\end{split}
\end{equation}
It follows from Theorem \ref{thm1} that
\begin{equation}\label{eqn:theorem 6 fault error1}
    \lim\limits_{t \to \infty }{\left\Vert f(t)-\hat{f}(t)\right\Vert}=0.
\end{equation}
Hence, Theorem \ref{theorem:6} is proven.
\end{proof}
\section{Application to PEM Fuel Cell systems}
In this section, we demonstrate the observer design process using the reduced-order model of a fuel cell system which has been verified in \cite{talj2010experimental} experimentally.
The dynamic model of the fuel cell system is given as,
\begin{equation}\label{eqn:fuel cell model}
\left\{
\begin{split}
\dot{x}=&F(x)+g_uu+g_\xi\xi+g_ff,\\
y = &h(x),
\end{split}
\right.
\end{equation}
where the vector
$F(x)=
\begin{bmatrix}
f_1(x)&f_2(x)&f_3(x)&f_4(x)
\end{bmatrix}^T$
is
\begin{equation}\label{eqn:vector f}
\left\{
\begin{split}
f_1(x)=&-(c_1+c_8)(x_1-x_4)-\frac{c_3(x_1-c_2)\psi(x_1)}{\kappa x_1},\\
f_2(x)=&-c_8(x_1-x_4)-\frac{c_3x_2\psi(x_1)}{\kappa x_1},\\
f_3(x)=&-c_9x_3-\frac{c_{10}}{x_3}\left[(\frac{x_4}{c_{11}})^{c_{12}}-1\right]h_3(x_3),\\
f_4(x)=&c_{14}\left\{1+c_{15}\left[(\frac{x_4}{c_{11}})^{c_{12}}-1\right]\right\}
\times\left[h_3(x_3)+c_{16}(x_1-x_4)\right],
\end{split}
\right.
\end{equation}
and the input vectors are $g_u=
\begin{bmatrix}
0&0&c_{13}&0
\end{bmatrix}^T$, $g_\xi=
\begin{bmatrix}
-c_4&0&0&0
\end{bmatrix}^T$,
$g_f=
\begin{bmatrix}
0&0&0&c_5
\end{bmatrix}^T$.
The system states  $x=[x_1,x_2,x_3,x_4]^T$,
where $x_1$ and $x_2$ are the total pressure at the cathode and
nitrogen partial pressures respectively,
$x_3$ is the angular speed of the compressure, and
$x_4$ is the air pressure in the supply manifold;
The measurements of the system are
$y=h(x)=
\begin{bmatrix}
y_1&y_2
\end{bmatrix}
^T=
\begin{bmatrix}
x_1&x_4
\end{bmatrix}
^T$.
The stack current $\xi$ is considered as an uncertain parameter $\theta$ in our observer design, which eliminates the need of adding current sensors in the system. The control input $u$ represents the motor's quadratic current component is controlled
using a real time static feed-forward controller.

A failure of the fuel cell air circuit is considered, i.e. the pipe connecting the air-feed compressor to the fuel cell cathode explodes suddenly. This can be detected as a variation in supply manifold pressure dynamics \cite{escobet2009model}.
The fault signal $f(t)$ appears in the output channel $z_{y_2}$,
\begin{eqnarray}\label{fault signal f}
f(t) =
\begin{cases}
  3\times 10^{-3} \ \ kg/sec, \ \ if\ \ t\geq 50\ \ sec\\
  0, \quad\quad\quad\quad\quad\quad\quad\ \ else
\end{cases}
\end{eqnarray}
The function $\psi(x_1)$ is the total flow rate at the cathode exit,
\begin{eqnarray}\label{eqn:total flow rate}
\psi(x_1) =
\begin{cases}
  c_{17}x_1\left(\frac{c_{11}}{x_1}\right)^{c_{18}}\sqrt{1-\left(\frac{c_{11}}{x_1}\right)^{c_{12}}},
  \quad if \ \ \frac{c_{11}}{x_1}> c_{19}\\
  c_{20}x_1, \quad \quad\quad \quad \quad \quad \quad \quad\quad\ \ \ \ \ \  if \ \
  \frac{c_{11}}{x_1}\leq c_{19}
\end{cases}
\end{eqnarray}
and $h(x_3)$ is the mass flow rate of a twin screw compressor, where $h(x_3)=Ax_3$.
All the parameters $c_i$, $k_f$ are positive and depend on the physical values
of the fuel cell \cite{pukrushpan2004}(See Appendix A).

In order to design the proposed observer for the fuel cell system.
Let define $z_{y_1}:=x_1$, $z_{y_2}:=x_4$, $z_2:=\begin{bmatrix}x_2&x_3\end{bmatrix}^T$ and $\theta:=\xi$.
Then, system (\ref{eqn:fuel cell model}) is described as the form of (\ref{eqn:1})
\begin{equation}\label{sys:fc1}
\left\{
\begin{split}
\dot{z}_{y_1}=&-(c_1+c_8)(z_{y_1}-z_{y_2})-\frac{c_3(z_{y_1}-c_2)\psi(z_{y_1})}{\kappa z_{y_1}}
-c_4\theta,\\
\dot{z}_{y_2}=&c_{14}\left\{1+c_{15}\left[(\frac{z_{y_2}}{c_{11}})^{c_{12}}-1\right]\right\}
\times\left[h_3(D_2z_2)+c_{16}(z_{y_1}-z_{y_2})-f(t)\right],\\
\dot{z}_2=&
\underbrace{\begin{bmatrix}
-H&0\\
0&-c_9
\end{bmatrix}}_{A_{22}}
z_2+
\underbrace{\begin{bmatrix}
-c_8&c_8\\
0&0
\end{bmatrix}}_{A_{21}}
y+
\underbrace{\begin{bmatrix}
D_1z_2\left(H-\frac{c_3\psi(z_{y_1})}{\kappa z_{y_1}}\right)\\
-\frac{c_{10}}{D_2z_2}\left[\left(\frac{z_{y_2}}{c_{11}}\right)^{c_{12}}
-1\right]h_3(D_2z_2)+c_{13}u
\end{bmatrix}}_{g_2(y,z_2,u)},\\
y=&
\begin{bmatrix}
{z}_{y_1}\\
{z}_{y_2}
\end{bmatrix}.
\end{split}
\right.
\end{equation}
where
$D_1=\begin{bmatrix}1&0\end{bmatrix},
D_2=\begin{bmatrix}0&1\end{bmatrix},
W_{g_1}(y,z_2,u):=-\frac{c_3(z_{y_1}-c_2)\psi(z_{y_1})}{\kappa z_{y_1}},\\
W_{g_2}(y,z_2,u):=c_{14}c_{15}\left(\frac{z_{y_2}}{c_{11}}\right)^{c_{12}}
\left(h_3(D_2z_2)+c_{16}(z_{y_1}-z_{y_2})\right),
\Phi_1(y,u):=-c_4$, $\Phi_2(y,u):=c_5,
\phi_2(y,u)=0$
and the design gain parameter $H$ is chosen to
satisfy the Riccati equation (\ref{Riccati equation}). The fault signal is weighted by $c_5$, modeled as
\begin{equation}
c_5=-c_{14}\left[1+c_{15}\left(\left(\frac{z_{y_2}}{c_{11}}\right)^{c_{12}}-1\right)\right],
\end{equation}

The adaptive-gain STW observer for the system (\ref{sys:fc1})
is designed as the form (\ref{eqn:two output system observer}, \ref{eqn:adaptive laws})
\begin{equation}\label{sys:fc1obs}
\left\{
\begin{split}
\dot{\hat{z}}_{y_1}=&-(c_1+c_8)(z_{y_1}-z_{y_2})
-\frac{c_3(z_{y_1}-c_2)\psi(z_{y_1})}{\kappa z_{y_1}}
-c_4\hat{\theta}+\mu(e_{y_1}),\\
\dot{\hat{z}}_{y_2}=&c_{14}\left\{1+c_{15}
\left[(\frac{z_{y_2}}{c_{11}})^{c_{12}}-1\right]\right\}
\times\left[h_3(D_2\hat{z}_2)+c_{16}(z_{y_1}-z_{y_2})\right]+\mu(e_{y_2}),\\
\dot{\hat{z}}_2=&
\begin{bmatrix}
-c_8(z_{y_1}-z_{y_2})-\frac{c_3D_1\hat{z}_2\psi(z_{y_1})}{\kappa z_{y_1}}\\
-c_9D_2\hat{z}_2-\frac{c_{10}}{D_2z_2}\left[(\frac{z_{y_2}}{c_{11}})^{c_{12}}
-1\right]h_3(D_2\hat{z}_2)+c_{13}u
\end{bmatrix},\\
\hat{y}=&
\begin{bmatrix}
{\hat{z}}_{y_1}\\
{\hat{z}}_{y_2}
\end{bmatrix}.
\end{split}
\right.
\end{equation}
and
\begin{equation}\label{eqn:new adaptive parameter}
\begin{split}
\dot{\hat{\theta}}=&-K\left(
c_4\hat{\theta}+(c_1+c_8)(z_{y_1}-z_{y_2})
+\frac{c_3(z_{y_1}-c_2)\psi(z_{y_1})}{\kappa z_{y_1}}
+\dot{z}_{y_1}
\right).
\end{split}
\end{equation}
the adaptive-gains of the STW algorithm
$\mu(e_{y_1}), \mu(e_{y_2})$ are designed
according to (\ref{eqn:adaptive gains}, \ref{eqn:adaptive gains rt}).
The value of $\dot{z}_{y_1}$ is obtained from the robust exact finite time
differentiator (\ref{differentiator}) in \cite{levent1998robust}.
From (\ref{eqn:two estimation of f}), the fault signal $f(t)$ is estimated as $\hat{f}=\frac{\mu(e_{y_2})}{c_5}$.
\begin{remark}
The Assumptions (\ref{3.2},\ref{3.6}) are satisfied by the functions
$W_{g_1}(y,z_2,u)$, $W_{g_2}(y,z_2,u)$, $\Phi_1(y,u)$, $\Phi_2(y,u)$
and $g_2(y,z_2,u)$.
The Riccati equation in the Assumption \ref{3.3} will be satisfied by appropriate
value of the design gain $H>0$.
The Assumption \ref{3.4} is also satisfied for some $\epsilon>0$,
since the equation (\ref{matrix design parameter}) is simplified into a scalar equation.
\end{remark}
\section{Experimental Results}
Experiments have been performed on a Hardware-In-Loop (HIL) test bench shown in Fig. 1, which consists of a real time emulated fuel cell system and a twin screw compressor.
This emulated PEMFC stack is a 33 kW fuel cell composed of 90 cells in series,
which provides the cathode pressure and supply manifold pressure as outputs.
The twin screw compressor
consists of two helical rotors which are coupled directly to its motor.
Air intake is at the opposite side of the mechanical transmission and
the output pressure is regulated by a servo valve.
It has a flow rate margin 0-0.1 kg/s at a maximum velocity of 12000 RPM,
and is driven by a permanent synchronous motor (PMSM).
The lubrication system of this compressor is specifically designed to
prevent contamination of the air from lubricating oil.

The test bench is controlled by
the National Instruments CompactRIO realtime controller
and data acquisition system.
The motor quadratic current calculated by the real time controller in
$(d,q)$ coordinate is transformed by
an inverter between the CompactRIO and the compressor
to 3-phase coordinate in order to control the PMSM.
The fault reconstruction scheme and HIL simulation structure
is shown in Fig. 2.
The nominal value of the parameters for the simulation are shown in
Table 1.

The air mass-flow measured at the output of the compressor is fed
to the real time fuel cell emulation system, programmed in Labview.
The stack current shown in Fig. 7  is varied between 100A and 450A in order to deal with different load variation which corresponds the flow rate
variation between 0g/s and 28g/s.
The detail of the real time controller is available in
\cite{matraji2013}.

Figs. 3-6 show the states estimation of the system
(\ref{eqn:fuel cell model}).
Fig. 7 shows that the adaptive law gives a good estimate for the stack current which is
considered as an unknown parameter $\theta$.
Based on the estimates of the proposed observer,
the fault signal is reconstructed faithfully as shown in Fig. 8.
The time history of the adaptive-gain $L(t)$
is shown in Fig. 9, where the convergence of the proposed observer is ensured.
\section{CONCLUSION}
This paper has proposed a robust fault reconstruction method for
a class of nonlinear uncertain systems with Lipschitz nonlinearities
based on an adaptive STW sliding mode observer.
An adaptive update law has been given to identify the uncertain parameters.
An adaptive-gain STW observer is proposed to estimate the
system state variables exponentially even in the presence of uncertain
parameters and fault signals without requiring any information on the boundaries of the fault
and its time derivative.
Furthermore, the obtained equivalent output error injection
was then used to reconstruct the possible faults in the system.
The proposed fault reconstruction approach
was successfully implemented on a Hardware-In-Loop test bench which consists
of a commercial twin screw compressor and a real time fuel cell emulation system.
The fault signal represents as
a failure of the fuel cell air circuit, i.e. the pipe connecting the air-feed compressor to the fuel cell cathode explodes suddenly,
was reconstructed precisely.
The experimental results have shown that
the proposed approach is effective and feasible.
In future, other faults that affect fuel cell performance,
such as drying or flooding at the cell stack and starvation will
also be included.
\begin{figure}[H]
\begin{center}
\includegraphics[width=4in]{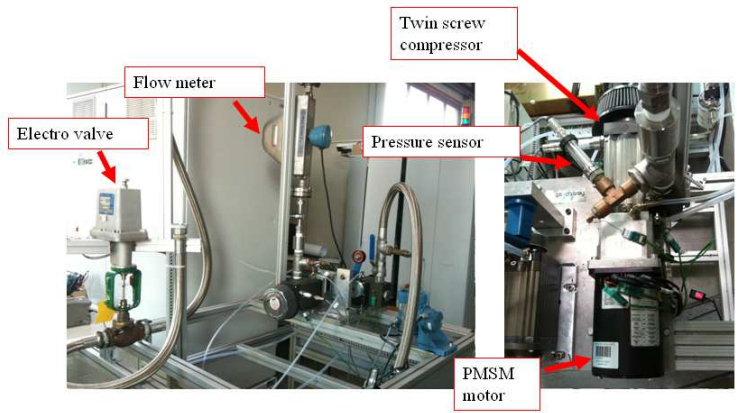}
\caption{Test Bench}
\label{fig:1}
\end{center}
\end{figure}
\begin{figure}[H]
\begin{center}
\includegraphics[width=4in]{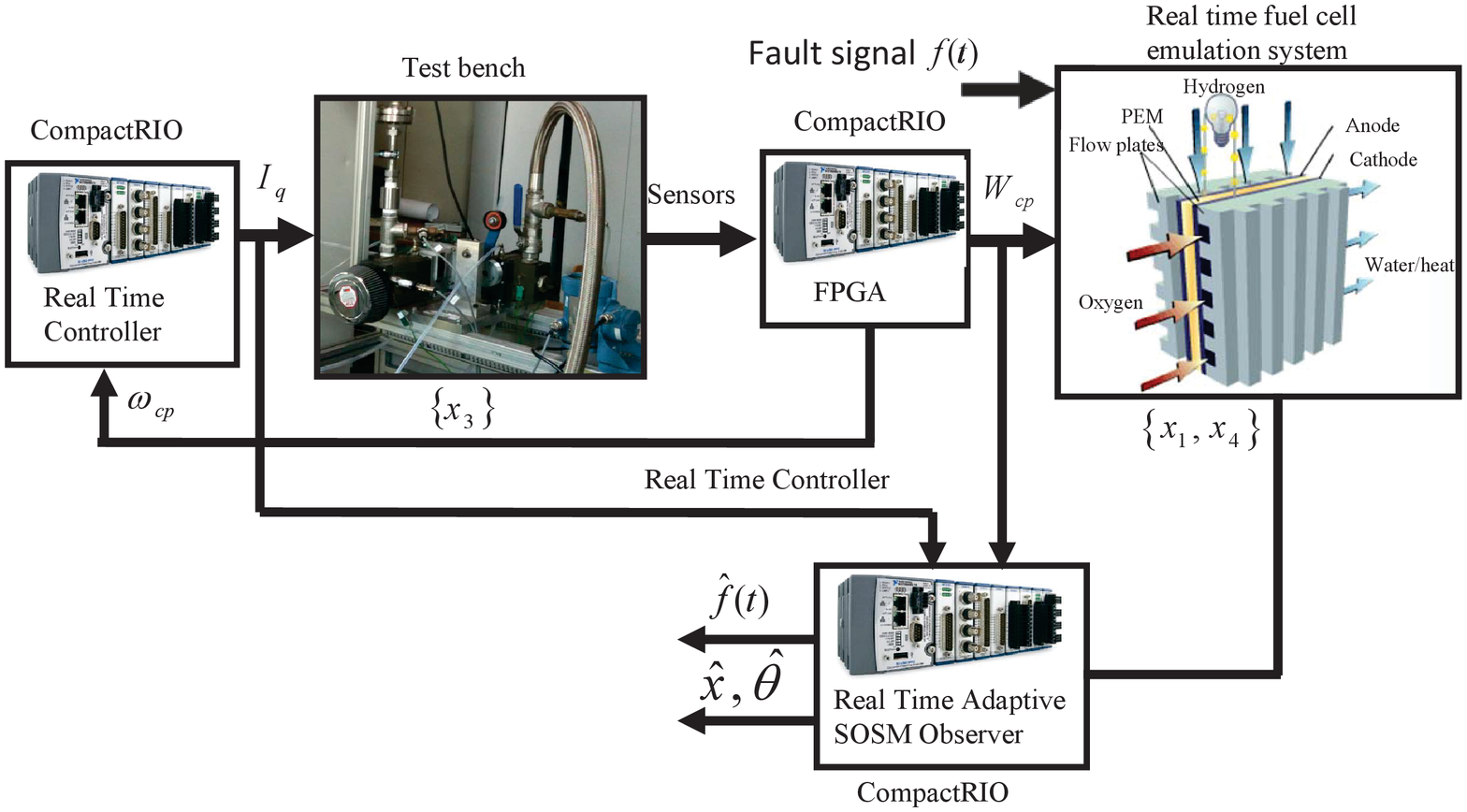}
\caption{Fault Reconstruction Scheme implemented on a Hardware-In-Loop Simulator}
\label{fig:2}
\end{center}
\end{figure}
\begin{table}[t]
\centering
\caption{\label{table:1}Parameters used in the simulation}
\begin{tabular}{ccc}
\hline
Symbol& Parameter & Value \\
\hline
$n$ &Number of cells in fuel cell stack & 90 \\
$R$ & Universal gas constant & 8.314 J/(mol K) \\
$R_a$ &Air gas constant & 286.9 J/(kg K) \\
$p_{atm}$ & Atmospheric pressure & 1.01325 $\times 10^5$ Pa \\
$T_{atm}$ & Atmospheric temperature & 298.15 K \\
$T_{fc}$ & Temperature of the fuel cell & 353.15 K \\
$F$ & faraday constant & 96485 C/mol \\
$M_a$ & Air molar mass & 28.9644$\times 10^{-3}$ kg/mol \\
$M_{O_2}$ & Oxygen molar mass & 32$\times 10^{-3}$ kg/mol \\
$M_{N_2}$ & Nitrogen molar mass & 28$\times 10^{-3}$ kg/mol \\
$M_v$ & Vapor molar mass & 18.02$\times 10^{-3}$ kg/mol \\
$C_D$ & Discharge of the nozzle & 0.0038 \\
$A_T$ & Operating area of the nozzle & 0.00138 $m^2$ \\
$\gamma$ & Ratio of specific heats of air & 1.4 \\
$J_{cp}$ & Compressor inertia & 671.9 $\times 10^{-5}$ kg $m^2$ \\
$f$ & Motor friction & 0.00136 V/(rad/s)\\
$k_t$ & Motor constant & 0.31 N m/A\\
$C_p$ & Constant pressure specific heat of air & 1004 J/(kg K)\\
$\eta_{cp}$ & Compressor efficiency & 80\%\\
$\eta_{cm}$ & Motor mechanical efficiency & 98\%\\
$V_{ca}$ & Cathode volume & 0.0015 $m^3$\\
$V_{sm}$ & Supply manifold volume & 0.003 $m^3$\\
$V_{cpr/tr}$ & Compressor volume per turn & 5$\times 10^{-4}m^3/tr$\\
$k_{ca,in}$ & Cathode inlet orifice constant & 0.3629$\times 10^{-5}kg/(Pa s)$\\
$k_{ca,out}$ & Cathode outlet orifice constant & 0.76$\times 10^{-4}kg/(Pa s)$\\
$\rho_a$ & Air density & 1.23 kg/$m^3$\\
$x_{O_2,ca,in}$ & Oxygen mass fraction & 0.23\\
$\mu$ & Smoothing filter time constant & 0.01 s\\
$L(0)$& Initial value of the adaptive gain L(t)&5000\\
$k$&Design parameter in (\ref{eqn:adaptive gains rt})&500\\
\hline
\end{tabular}
\end{table}
\begin{figure}[H]
\begin{center}
\includegraphics[width=4in]{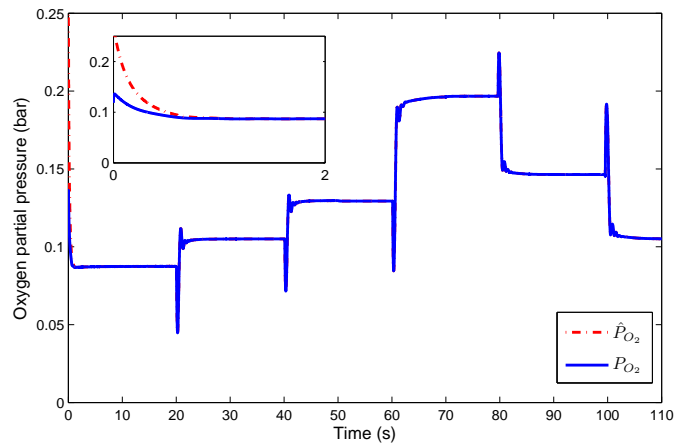}
\caption{Oxygen partial pressure estimation}
\label{fig:3}
\end{center}
\end{figure}
\begin{figure}[H]
\begin{center}
\includegraphics[width=4in]{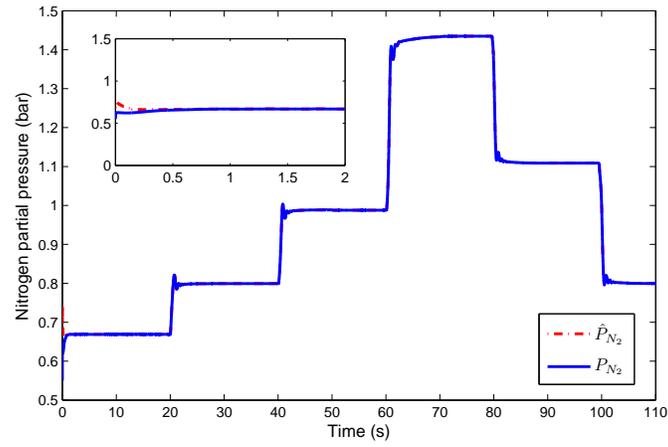}
\caption{Nitrogen partial pressure estimation}
\label{fig:4}
\end{center}
\end{figure}
\begin{figure}[H]
\begin{center}
\includegraphics[width=4in]{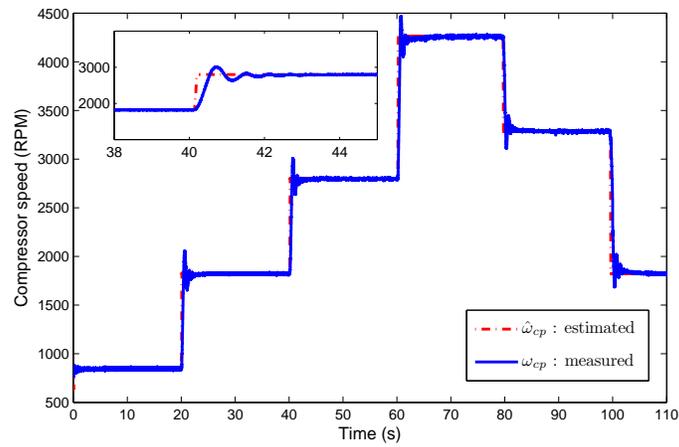}
\caption{Compressure speed estimation}
\label{fig:5}
\end{center}
\end{figure}
\begin{figure}[H]
\begin{center}
\includegraphics[width=4in]{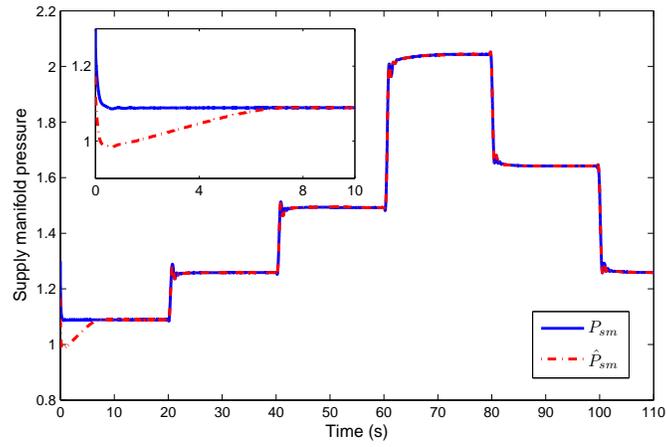}
\caption{Supply manifold pressure estimation}
\label{fig:6}
\end{center}
\end{figure}
\begin{figure}[H]
\begin{center}
\includegraphics[width=4in]{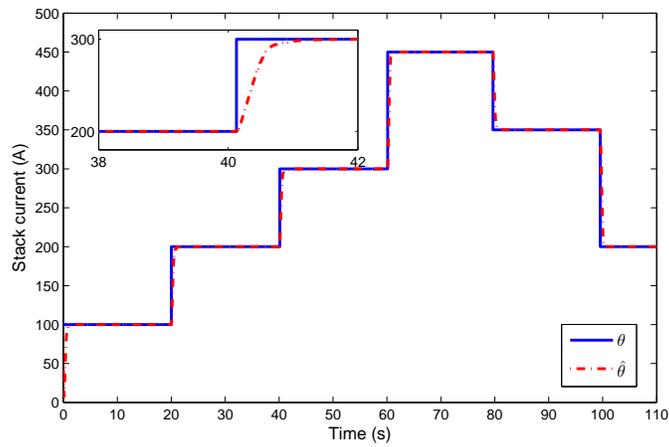}
\caption{Adaptive estimation of the stack current $\xi$}
\label{fig:7}
\end{center}
\end{figure}
\begin{figure}[H]
\begin{center}
\includegraphics[width=4in]{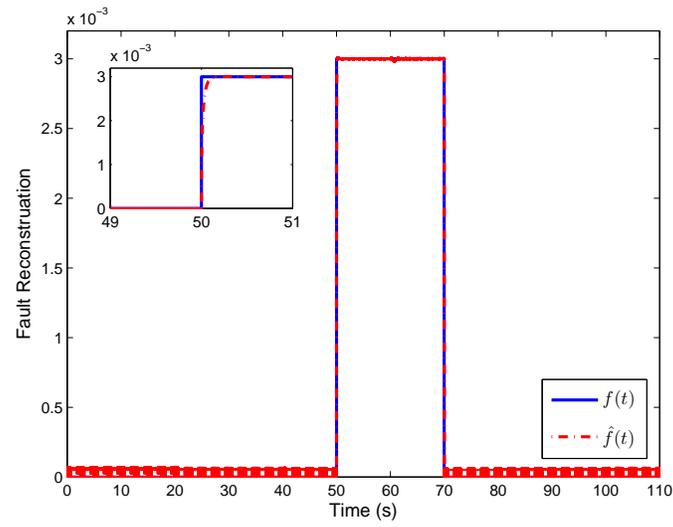}
\caption{Fault reconstruction $f$ and $\hat{f}$}
\label{fig:8}
\end{center}
\end{figure}
\begin{figure}[H]
\begin{center}
\includegraphics[width=4in]{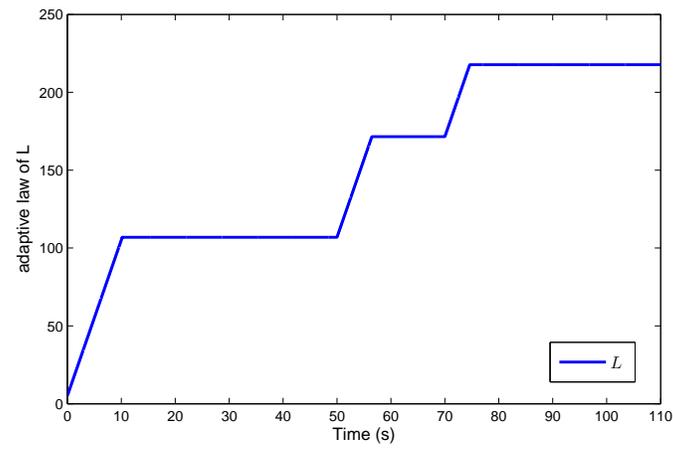}
\caption{Adaptive law of $L(t)$}
\label{fig:9}
\end{center}
\end{figure}
\appendix
\section{}
\begin{tabular}[H]{ll}
$c_1=\frac{\bar{R}T_{st}k_{ca,in}}{M_{O_2}V_{ca}}
\left(\frac{x_{O_2,atm}}{1+\omega_{atm}}\right)$;&$c_2=P_{sat}$\\
$c_3=\frac{\bar{R}T_{st}}{V_{ca}}$;&$c_4=\frac{\bar{R}T_{st}n}{4V_{ca}F}$\\
$c_8=\frac{\bar{R}T_{st}k_{ca,in}}{M_{N_2}V_{ca}}
\left(\frac{1-x_{O_2,atm}}{1+\omega_{atm}}\right)$;\\
$c_9=\frac{\eta_{cm}k_t}{J_{cp}}$;&$c_{10}=\frac{C_pT_{atm}}{J_{cp}\eta_{cp}}$\\
$c_{11}=P_{atm}$;&$c_{12}=\frac{\gamma-1}{\gamma}$\\
$c_{13}=\frac{\eta_{cm}k_t}{J_{cp}}$;&$c_{14}=
\frac{\gamma\bar{R}T_{atm}}{M_{a,atm}V_{sm}}$\\
$c_{15}=\frac{1}{\eta_{cp}}$;&$c_{16}=k_{ca,in}$\\
$c_{17}=\frac{C_DA_T}{\sqrt{\bar{R}T_{st}}}
\sqrt{\frac{2\gamma}{\gamma-1}}$;&$c_{18}=\frac{1}{\gamma}$\\
$A=\frac{1}{2\pi}\eta_{\upsilon-c}V_{cpr/tr}\rho_a$;&$c_{19}=\left( \frac{2}{\gamma+1}\right)^{\frac{\gamma}{\gamma-1}}$\\
$c_{20}=\frac{C_DA_t}{\sqrt{\bar{R}T_{st}}}\gamma^{\frac{1}{2}}\left( \frac{2}{\gamma+1}\right)^{\frac{\gamma+1}{2(\gamma-1)}}$.
\end{tabular}
\bibliographystyle{IEEEtran}
\bibliography{adaptive_stw_obs}
\end{document}